\documentclass[11pt]{amsart}

\usepackage{amsmath,amssymb,amsthm,mathrsfs}
\usepackage{mathpazo}
\usepackage{eulervm}
\usepackage[margin=1in]{geometry}
\usepackage{hyperref}
\usepackage{enumerate}

\newtheorem{theorem}{Theorem}[section]
\newtheorem{lemma}[theorem]{Lemma}
\newtheorem{proposition}[theorem]{Proposition}
\newtheorem{corollary}[theorem]{Corollary}
\newtheorem{fact}[theorem]{Fact}
\newtheorem{condition}[theorem]{Condition}

\newtheorem*{theoremstar}{Theorem}

\theoremstyle{definition}
\newtheorem{definition}[theorem]{Definition}

\theoremstyle{remark}
\newtheorem{remark}[theorem]{Remark}

\DeclareMathOperator{\Ind}{Ind}
\DeclareMathOperator{\ind}{ind}
\DeclareMathOperator{\Hom}{Hom}
\DeclareMathOperator{\sgn}{sgn}
\DeclareMathOperator{\Tr}{Tr}

\DeclareMathOperator{\Res}{Res}
\DeclareMathOperator{\vol}{vol}
\DeclareMathOperator{\Gal}{Gal}

\newcommand{\CC}{\mathbb{C}}
\newcommand{\ZZ}{\mathbb{Z}}

\newcommand{\HH}{\mathcal{H}}
\newcommand{\AAA}{\mathcal{A}}
\newcommand{\OO}{\mathcal{O}}
\newcommand{\pp}{\mathfrak{p}}
\newcommand{\barr}[1]{\overline{#1}}

\title[Iwahori component of the Gelfand--Graev for reductive groups]{Iwahori component of the Gelfand--Graev representation \\ for reductive groups}

\author{Yi Luo}
\address{Department of Mathematics, University of Utah, Salt Lake City, UT 84112}
\email{yiluo@math.utah.edu}
\date{20 July 2026}

\begin{document}

\begin{abstract}
Let $G$ be a connected reductive group over a $p$-adic field~$F$, $U$ the unipotent radical of a minimal parabolic subgroup, $\psi$ a depth-zero non-degenerate character of $U(F)$, and $I$ an Iwahori subgroup of $G(F)$. We give a direct calculation of the $I$-fixed vectors in the Gelfand--Graev representation $\ind_U^G\psi$. The function $\mathrm{ch}_I^\psi$ supported on $U(F)I$ and normalized at the identity freely generates this space over the Bernstein subalgebra~$\AAA$ and transforms under the finite Hecke subalgebra $\HH_{W_0}$ by its sign representation~$\sgn$. Consequently, $(\ind_U^G\psi)^I \cong \HH \otimes_{\HH_{W_0}} \sgn$ as modules over the Iwahori--Hecke algebra~$\HH$. This extends the explicit construction of Chan--Savin from split groups to all connected reductive groups.
\end{abstract}

\maketitle

%%% ------------------------------------------------------------------
\section{Introduction}
%%% ------------------------------------------------------------------

Let $F$ be a finite extension of~$\mathbb{Q}_p$. We write $\OO_F$ for its ring of integers, $\varpi$ for a uniformizer, $\pp = \varpi\OO_F$ for the maximal ideal, and $k_F = \OO_F/\pp$ for the residue field, of cardinality~$q$. Let $G$ be a connected reductive group defined over~$F$. Fix a maximal $F$-split torus $S \subset G$ and let $M = Z_G(S)$ be its centralizer. Then $M$ is the Levi factor of a minimal parabolic $F$-subgroup $P = MU$ with unipotent radical~$U$, and $M$ is reductive and anisotropic modulo its center. Let $\barr{P} = M\barr{U}$ be the opposite minimal parabolic and let $W_0 = N_G(S)(F)/M(F)$ be the relative Weyl group. Let $\psi$ be a depth-zero non-degenerate character of $U(F)$ at a special vertex~$v$ in the apartment of~$S$, in the sense of Section~\ref{sec:GG}, and let $I \subset G(F)$ be the Iwahori attached to $(P,v)$ as in Section~\ref{ssec:Iwahori}.

Let $\HH = C_c(I\backslash G(F)/I)$ be the Iwahori--Hecke algebra. By the Bernstein presentation of Rostami~\cite{Rostami}, the $\CC$-algebra $\HH$ contains a commutative subalgebra $\AAA$ and a finite subalgebra $\HH_{W_0}$ indexed by the relative Weyl group. Moreover $\HH$ is free as a right $\HH_{W_0}$-module with basis the Bernstein elements $\{\Theta_\lambda : \lambda \in \Omega_M\}$, which span~$\AAA$. The algebra $\HH_{W_0}$ has a one-dimensional sign representation $\sgn \colon T_w \mapsto (-1)^{\ell(w)}$. Our main result is the following.

Let $\mathrm{ch}_I^\psi$ be the $I$-fixed function in $\ind_U^G\psi$ supported on $U(F)I$ and defined by $\mathrm{ch}_I^\psi(u i)=\psi(u)$ for $u\in U(F)$ and $i\in I$.

\begin{theoremstar}[{= Corollary~\ref{cor:main}}]
The map
\[
\HH \otimes_{\HH_{W_0}} \sgn \longrightarrow (\ind_U^G \psi)^I,
\qquad h\otimes 1\longmapsto h\cdot\mathrm{ch}_I^\psi,
\]
is an isomorphism of $\HH$-modules. In particular, $(\ind_U^G \psi)^I$ is a free $\AAA$-module of rank one generated by $\mathrm{ch}_I^\psi$.
\end{theoremstar}

When $G$ is split this is the theorem of Chan--Savin~\cite{CS}. Bushnell and Henniart~\cite{BH} proved that every Bernstein component of $\ind_U^G\psi$ is finitely generated. Mishra and Pattanayak~\cite{MP} refined this for the principal-series blocks $[M,\chi]$ attached to characters $\chi$ of the minimal Levi. Whenever such a block admits a Bushnell--Kutzko type $(K,\rho)$, the $\rho$-isotypic component of $\ind_U^G\psi$ is a cyclic module over the Hecke algebra $\HH(G,\rho)$ (\cite[Theorem~1]{MP}). Such a type exists by Fintzen~\cite{Fintzen} when $G$ splits over a tamely ramified extension and $p \nmid |W|$, where $W$ is the absolute Weyl group. For \emph{split} $G$ and depth-zero generic~$\psi$ they identify this component with a sign-induced module $\HH(G,\rho)\otimes_{\HH_{W,\chi}}\sgn$ (\cite[Theorem~3]{MP}), under auxiliary hypotheses that depend on the block. For the Iwahori block $\chi = 1$ their theorem carries no such hypothesis and recovers the theorem of Chan--Savin.

A broader framework is provided by Solleveld~\cite[Theorem~2.7 and Lemma~3.1]{Sol25}, who obtains the analogous sign-induced description for the Hom-space from a Bernstein progenerator to $\ind_U^G\psi$ in every principal-series block of a quasi-split group. Solleveld and Opdam~\cite[Theorem~6.1(b), Proposition~6.2, Corollary~8.7, and Theorem~A.1]{SO26} extend this description and the resulting genericity criterion to simply generic Bernstein blocks of an arbitrary connected reductive group. Up to specialization to the relevant Iwahori blocks and the appropriate comparisons of progenerators and Hecke-algebra normalizations, their work contains the abstract sign-induced description underlying the theorem above.
Appendix~\ref{app:Sol} recovers $\mathrm{ch}_I^\psi$ from their distinguished vectors by Fourier inversion.

Our aim here is more concrete. We work directly with $(\ind_U^G\psi)^I$, identify its generator with the normalized function $\mathrm{ch}_I^\psi$, and compute the action of the Iwahori--Matsumoto generators on it. This calculation is needed in~\cite{Luo26}, where applying the spherical idempotent to $\mathrm{ch}_I^\psi$ leads to the Casselman--Shalika formula. We follow the method of Chan--Savin~\cite{CS}. Three observations make the general case as clean as the split one. First, since the minimal Levi $M = Z_G(S)$ is anisotropic modulo center, its Iwahori--Weyl group is the finitely generated abelian group $\Omega_M$ with no finite Weyl part (\cite[\S 2.9]{Rostami}). The Bernstein subalgebra $\AAA \cong \CC[\Omega_M]$ is therefore commutative even when $M$ is not a torus. The space $(\ind_U^G\psi)^I$ is a free $\AAA$-module of rank one. Second, the transformation law $T_s \cdot \mathrm{ch}_I^\psi = -\mathrm{ch}_I^\psi$ reduces to a character sum over the reductive quotient $\mathsf{G}_v = K_v/K_v^+$ of the special maximal parahoric at~$v$. By Lang's theorem the group $\mathsf{G}_v$ is quasi-split over~$k_F$. Its relative rank-one subgroups are of type $\mathrm{SL}_2$ or $\mathrm{SU}_3$, and the two character sums both equal $-1$ (Proposition~\ref{prop:ranksum}). The $\mathrm{SL}_2$ sum is the computation of Chan--Savin, and the $\mathrm{SU}_3$ sum is new. Third, a depth-zero non-degenerate $\psi$ exists for every connected reductive~$G$. Indeed, the classical equivalence between genericity and quasi-splitness concerns the unipotent radical of a \emph{Borel} subgroup of~$G$. Depth-zero non-degeneracy concerns instead the unipotent radical of a Borel subgroup of the quasi-split reductive quotient~$\mathsf{G}_v$, so the required character always exists (Lemma~\ref{lem:exists}).

The depth-zero non-degenerate characters of $U(F)$ may form several $S(F)$-orbits, and the theorem holds for each orbit. The calculation also yields a genericity criterion. A smooth representation $\pi$ generated by its $I$-fixed vectors is $\psi$-generic if and only if $\pi^I$ contains a nonzero vector transforming by~$\sgn$ (Corollary~\ref{cor:genericity}).

Throughout, $V^I$ (or $\pi^I$) denotes the $I$-fixed vectors of a smooth representation $(\pi,V)$ and $V_{\barr{U}}$ the normalized Jacquet module with respect to~$\barr{U}$. Haar measures are normalized so that $\vol(I) = 1$, $\vol(U(F)\cap I) = 1$, and $\vol(\barr{U}(F)\cap I) = 1$.

%%% ------------------------------------------------------------------
\section*{Acknowledgements}
%%% ------------------------------------------------------------------

The author thanks Professor Gordan Savin for explaining the compatible choice of the Iwahori subgroup and the character. The author also thanks Professor Maarten Solleveld for bringing his joint work with Eric Opdam to the author's attention. The author was partially supported by a C.~R.~Wiley Instructorship at the University of Utah.

%%% ------------------------------------------------------------------
\section*{Tool and computational resource disclosure}
%%% ------------------------------------------------------------------

When preparing this paper, we used AI tools at two levels. At the model level, we used the Anthropic large language models Claude Opus 4.8 and Claude Fable 5. At the agent level, we used Claude Code and a custom agent built by the author for mathematical research. These tools drafted parts of the text, revised the exposition, and checked the citations. The human author verified all mathematical content and takes full responsibility for it.

%%% ------------------------------------------------------------------
\section{Preliminaries}\label{sec:prelim}
%%% ------------------------------------------------------------------

\subsection{The relative root system}\label{ssec:structure}

Fix a maximal $F$-split torus $S \subset G$. Let $\Phi = \Phi(G,S)$ be the relative root system, $\Phi^+$ the positive system attached to~$P$, and $\Delta \subset \Phi^+$ the simple roots. For a non-divisible $\alpha \in \Phi$ let $U_\alpha$ be the relative root group. The root group $U_\alpha$ is abelian unless $2\alpha \in \Phi$, in which case it is a two-step unipotent group with central subgroup $U_{2\alpha}$ and abelian quotient $U_\alpha/U_{2\alpha}$. Let $\Phi^+_{\mathrm{nd}}$ be the set of non-divisible positive roots. The unipotent radical decomposes as $U = \prod_{\alpha \in \Phi^+_{\mathrm{nd}}} U_\alpha$ in any fixed order, and $\barr{U} = \prod_{\alpha \in \Phi^+_{\mathrm{nd}}} U_{-\alpha}$. 

The centralizer $M = Z_G(S)$ is a connected reductive group and is the Levi factor of the minimal parabolic $P = MU$. Its maximal split torus is the central torus~$S$, so $M$ is anisotropic modulo its center. Hence $M(F)$ is compact modulo the center $Z(M)(F)$. Since $M_{\mathrm{der}}$ is then anisotropic over~$F$, the reduced building of $M$ over~$F$ is a single point.

The relative Weyl group $W_0 = N_G(S)(F)/M(F)$ is a finite Coxeter group with generating set the simple reflections $s_\alpha$, $\alpha \in \Delta$.

\subsection{The Iwahori subgroup and a special vertex}\label{ssec:Iwahori}

Let $\mathcal{B}$ be the Bruhat--Tits building of the adjoint group, and let $v \in \mathcal{B}$ be a special vertex in the apartment of~$S$. When $2\alpha\in\Phi$ for a simple root~$\alpha$, we say that $\alpha$ \emph{reduces non-centrally} at~$v$ if the depth-zero quotient $U_{\alpha,0}/U_{\alpha,0+}$ of the Moy--Prasad subgroups of $U_\alpha(F)$ at~$v$ (\cite{MoyPrasad94}, recalled below) has non-zero image in the abelianization $U_\alpha/U_{2\alpha}$, rather than collapsing into $U_{2\alpha}$. We impose:

\begin{condition}\label{cond:noncentral}
Every multipliable simple root reduces non-centrally at~$v$.
\end{condition}

The condition is vacuous when $\Phi$ is reduced. A depth-zero non-degenerate character of $U(F)$ is trivial on $U_{2\alpha}\subseteq[U,U]$, so it can exist at~$v$ only under Condition~\ref{cond:noncentral}. Such a vertex exists for every connected reductive~$G$ (Lemma~\ref{lem:vertex}).

Let $K_v = \mathcal{G}_v^\circ(\OO_F)$ be the associated special maximal parahoric subgroup, where $\mathcal{G}_v^\circ$ is the parahoric group scheme associated to~$v$ (\cite[\S8]{KP}). Equivalently, by \cite[Def.~1, Prop.~3]{HRpar}, we have $K_v = G(F)_v\cap\ker\kappa_G$, where $\kappa_G$ is the Kottwitz homomorphism~\cite{Kottwitz}. Its pro-unipotent radical is $K_v^+ = G(F)_{v,0+}$ and its reductive quotient is
\[
K_v/K_v^+ = \mathsf{G}_v(k_{F}), 
\]
where $\mathsf{G}_v$ is a connected reductive group over~$k_F$ (\cite[\S 13.2.6]{KP}), discussed in Section~\ref{sec:finite}. Fix the alcove $\mathfrak{a}$ in the apartment of~$S$ whose closure contains~$v$ and which is anti-dominant for~$P$. Let~$I$ be the Iwahori subgroup associated to~$\mathfrak{a}$. Then $I \subseteq K_v$ and $K_v^+ \subseteq I$. Set $M_1 = I\cap M(F)$. By the Levi compatibility of parahoric subgroups (\cite[Cor.~9.7.3]{KP}, applied to $\Omega = \mathfrak a$ and~$\{v\}$),
\[
M_1 = I\cap M(F) = K_v\cap M(F)
\]
is the connected parahoric of~$M$ at the unique point of its reduced building.  Since $v$ is special, the image $\mathsf{S}$ of~$S$ is a maximal split torus of $\mathsf{G}_v$, and the reductive quotient $\mathsf{M}$ of $M_1$ is the centralizer
$Z_{\mathsf{G}_v}(\mathsf{S})$.   Since the maximal split torus of $M$ is the central~$S$, the group $\mathsf{M}$ is connected reductive with central maximal split torus. Since $\mathsf{M}$ is quasi-split by Lang's theorem, it must be a torus. Moreover, $\mathsf{M} = Z_{\mathsf{G}_v}(\mathsf{S})$ is a maximal torus of~$\mathsf{G}_v$, because it contains every maximal torus containing~$\mathsf{S}$. Let $\mathsf{U}$ and $\barr{\mathsf{U}}$ be the images of $U$ and $\barr{U}$, the unipotent radicals of two opposite minimal parabolic subgroups of $\mathsf{G}_v$ with common Levi~$\mathsf{M}$. Since $\mathsf{M}$ is a maximal torus, the minimal parabolics $\mathsf{B} = \mathsf{M}\mathsf{U}$ and $\barr{\mathsf{B}} = \mathsf{M}\barr{\mathsf{U}}$ are opposite Borel subgroups of~$\mathsf{G}_v$. Because $\mathfrak{a}$ is anti-dominant for~$P$, the Iwahori $I$ is precisely the preimage of the \emph{opposite} Borel $\barr{\mathsf{B}}(k_F)$ under $K_v \to \mathsf{G}_v(k_F)$.

Consider the Moy--Prasad filtration subgroups $U_{\alpha,0} = U_\alpha(F)\cap K_v$ and $U_{\alpha,0+} = U_\alpha(F)\cap K_v^+$ of $U_\alpha(F)$. We write $U_0$, $U_{0+}$, and $\barr{U}_0$ for the analogous subgroups of the full unipotent radicals. By direct spanning, $U_0 = \prod_{\alpha\in\Phi^+_{\mathrm{nd}}} U_{\alpha,0}$ in any fixed order, and likewise for $U_{0+}$ and $\barr{U}_0$. The reduction map $K_v \to \mathsf{G}_v(k_F)$ induces $U_0/U_{0+} \xrightarrow{\ \sim\ } \mathsf{U}(k_F)$. We have the Iwahori factorization (\cite[Cor.~7.4.9]{KP}):
\begin{equation}\label{eq:iwahori}
I = (I\cap U)\cdot(I\cap M)\cdot(I\cap\barr{U}) = U_{0+}\cdot M_1\cdot\barr{U}_{0}.
\end{equation}

Since $v$ is special, the reduction map $N_{K_v}(S) \to W_0$ is surjective (\cite[Prop.~9.9.1]{KP}). We fix, for each $w \in W_0$, a representative $\dot{w} \in N_{K_v}(S)$. Each $\dot{w}$ has trivial translation part in the Iwahori--Weyl group of Section~\ref{ssec:Hecke} and preserves the Moy--Prasad filtration at~$v$, so that $\dot{w}\,U_{\beta,r}\,\dot{w}^{-1} = U_{w\beta,r}$ for every relative root~$\beta$ and $r \in \{0, 0+\}$. In particular $\dot{s}_\alpha$ conjugates $U_{-\alpha,0}$ to $U_{\alpha,0}$.

Let $\pi_1(M)$ be the algebraic fundamental group of~$M$, and set
\[
\Omega_M := \pi_1(M)_{\mathrm{In}}^{\mathrm{Fr}},
\]
the coinvariants under the inertia group~$\mathrm{In}$, followed by the invariants under the Frobenius. This is a finitely generated abelian group. The Kottwitz homomorphism $\kappa_M: M(F) \to \Omega_{M}$ is surjective (\cite[Cor.~11.7.5]{KP}).

\begin{lemma}\label{lem:M1}
The kernel of $\kappa_M$ is the unique parahoric subgroup of $M(F)$, which coincides with $M_1 = I\cap M(F)$. Hence $M_1$ is normal in $M(F)$ and $\kappa_M$ induces an isomorphism
\[
M(F)/M_1 \xrightarrow{\ \sim\ } \Omega_M.
\]
\end{lemma}

\begin{proof}
On the one hand, since $M$ is anisotropic modulo center, its reduced building is a point and $M(F)$ has a \emph{unique} parahoric subgroup, namely $\ker\kappa_M$ (\cite[Prop.~11.5.4]{KP}). On the other hand, as recalled above, by the Levi compatibility of parahoric subgroups, $M_1 = I\cap M(F)$ is a parahoric subgroup of $M(F)$, hence must coincide with the unique parahoric $\ker\kappa_M$. Therefore, $M_1 = \ker\kappa_M$ is normal, and the surjection $\kappa_M$ induces the isomorphism $M(F)/M_1\xrightarrow{\ \sim\ }\Omega_M$.
\end{proof}

\subsection{The Iwahori--Hecke algebra and its Bernstein presentation}\label{ssec:Hecke}

The Bruhat decomposition $G(F) = \bigsqcup_{w\in\widetilde{W}} IwI$ holds, indexed by the Iwahori--Weyl group~$\widetilde{W} = N_G(S)(F)/M_1$ (\cite[Thm.~7.8.1]{KP}). The Iwahori--Hecke algebra $\HH = C_c(I\backslash G(F)/I)$ accordingly has $\CC$-basis the characteristic functions $T_w$ of these double cosets, $w \in \widetilde{W}$. Let $\ell \colon \widetilde{W} \to \ZZ_{\ge 0}$ be the length function and $q_w = [IwI:I]$. The Iwahori--Matsumoto relations hold:
\begin{align*}
T_{w_1}T_{w_2} &= T_{w_1 w_2} && \text{if } \ell(w_1 w_2) = \ell(w_1)+\ell(w_2),\\
(T_s - q_s)(T_s &+ 1) = 0 && \text{for each simple affine reflection } s.
\end{align*}

For $f\in\HH$ and any smooth representation $(\pi,V)$ of $G(F)$, the $I$-invariant subspace $V^I$ becomes a left $\HH$-module under the action $\pi(f)\,v = \int_{G(F)} f(g)\,\pi(g)\,v\,\mathrm{d}g$. Since $\vol(I) = 1$, we note
\begin{equation}\label{eq:Haction}
T_w\cdot v = \sum_{x\in IwI/I} \pi(x)\,v
\end{equation}
over the left cosets $xI\subseteq IwI$. When $V$ is a space of functions on which $G(F)$ acts by right translation, the above action is right convolution $\pi(f)v = v * \hat f$ with
$\hat f(g) = f(g^{-1})$. In particular, \eqref{eq:Haction} reads $(T_w\cdot f)(g) = \sum_{x} f(g\,x)$.

By~\cite[\S 2.9]{Rostami} there is a semidirect product decomposition $\widetilde{W} = \Omega_M \rtimes W_0$. The finitely generated abelian group $\Omega_M$ embeds as the normal translation subgroup of $\widetilde{W}$, and this embedding identifies $M(F)/M_1$ with $\Omega_M$ through the Kottwitz homomorphism $\kappa_M$.

\smallskip
\noindent\textbf{Bernstein presentation.} By~\cite[\S 5]{Rostami} the algebra $\HH$ contains \emph{Bernstein elements} $\Theta_\lambda$, $\lambda \in \Omega_M$, such that the assignment $\Theta_\lambda \mapsto e^\lambda$ identifies their span $\AAA = \bigoplus_{\lambda} \CC\,\Theta_\lambda$ with the group algebra $\CC[\Omega_M]$ as a $\CC$-algebra. Moreover the set $\{\Theta_\lambda\, T_w : \lambda \in \Omega_M,\, w \in W_0\}$ is a $\CC$-basis of~$\HH$. Let $\HH_{W_0}$ be the subalgebra of $\HH$ spanned by $\{T_w : w \in W_0\}$. Its generators $T_{s}$, $s = s_\alpha$ with $\alpha \in \Delta$, satisfy $(T_s - q_s)(T_s + 1) = 0$. The assignment $T_s \mapsto -1$ extends to a one-dimensional representation:
\[
\sgn \colon \HH_{W_0} \to \CC, \qquad T_w \mapsto (-1)^{\ell(w)}.
\]
This basis exhibits $\HH$ as a free right $\HH_{W_0}$-module with basis $\{\Theta_\lambda : \lambda \in \Omega_M\}$, so $\HH \otimes_{\HH_{W_0}} \sgn$ is a free $\AAA$-module of rank one generated by $1 \otimes 1$. The presentation is completed by the Bernstein relations, which express the differences $T_s\,\Theta_\lambda - \Theta_{s(\lambda)}\,T_s$, for simple reflections $s \in W_0$, as explicit elements of~$\AAA$. These relations are not needed in this paper, and we refer the reader to~\cite[\S 5.4]{Rostami} for their precise form.

\begin{remark}\label{rem:comm}
Because $M$ is anisotropic modulo center, $\Omega_M$ is abelian and $\AAA \cong \CC[\Omega_M]$ is a \emph{commutative} algebra, even when $M$ is not a torus. Since $M_1$ is normal in $M(F)$ with quotient $\Omega_M$ (Lemma~\ref{lem:M1}), once we normalize the measure such that $\vol(M_1) = 1$, the parahoric Hecke algebra $C_c(M_1\backslash M(F)/M_1)$ is nothing but the group algebra $\CC[\Omega_M]$. The assignment $\mathrm{ch}_{M_1\lambda} \mapsto \Theta_\lambda$ is an embedding of $C_c(M_1\backslash M(F)/M_1)$ into~$\HH$ (cf.\ \cite[\S 3]{Bushnell01}). It coincides with the normalized embedding of \cite[Rem.~6]{Bushnell01} attached to~$\barr{P}$. Indeed, both maps are algebra homomorphisms sending $\mathrm{ch}_{M_1\lambda}$ to $q_\lambda^{-1/2}\,T_\lambda$ for every $(\barr{P},I)$-positive $\lambda$ in the sense of \cite[\S 3.1]{Bushnell01}. Such $\lambda$ generate the group~$\Omega_M$.
\end{remark}

\medskip
\noindent\textbf{The Iwasawa decomposition and the Bruhat lemma.} Recall the Iwasawa-type decomposition
\begin{equation}\label{eq:iwasawa}
G(F) = \bigsqcup_{w\in\widetilde{W}} U(F)\,\dot{w}\,I
\end{equation}
(\cite[Thms.~5.3.3, 7.8.1]{KP}), where $\dot{w}\in N_G(S)(F)$ is any lift of~$w$. Thus a left $(U,\psi)$-equivariant, right $I$-invariant function on $G(F)$ is determined by its values at the representatives~$\dot{w}$. The next lemma records how the generators $T_s$ move these cells. Both parts are used in Section~\ref{sec:GG}.

\begin{lemma}\label{lem:bruhat}
Let $s = s_\alpha$ be a simple reflection, $\alpha\in\Delta$, and write $\barr{w}\in W_0$ for the image of $w\in\widetilde{W}$.
\begin{enumerate}[\upshape(1)]
\item \textup{(Coset decomposition.)} $\displaystyle I s I = \bigsqcup_{y} y\,\dot{s}\,I$, where $y$ runs over~$q_s$ representatives of $U_{-\alpha,0}/U_{-\alpha,0+}$.
\item \textup{(Bruhat lemma.)}
\[
U(F)\,w\,IsI = \begin{cases}
U(F)\,ws\,I & \text{if } \barr{w}(\alpha) \in \Phi^-,\\
U(F)\,ws\,I \cup U(F)\,w\,I & \text{if } \barr{w}(\alpha) \in \Phi^+.
\end{cases}
\]
\end{enumerate}
\end{lemma}

\begin{proof}
(1) By~\eqref{eq:iwahori}, $I = \prod_{\beta\in\Phi^+_{\mathrm{nd}}} U_{-\beta,0}\cdot M_1\cdot \prod_{\beta\in\Phi^+_{\mathrm{nd}}} U_{\beta,0+}$.  Conjugation by $\dot{s} = \dot{s}_\alpha$ sends $U_{\gamma,r}$ to $U_{s_\alpha\gamma,r}$. It therefore carries every factor of $I$ other than $U_{-\alpha,0}$ into $I$. The factor $U_{-\alpha,0}$ is sent to $U_{\alpha,0}\not\subseteq I$. Hence $I\cap\dot{s}I\dot{s}^{-1}$ is $I$ with $U_{-\alpha,0}$ shrunk to $U_{-\alpha,0+}$, and $I = U_{-\alpha,0}\cdot(I\cap\dot{s}I\dot{s}^{-1})$, where the representatives $y$ may be taken in $U_{-\alpha,0}$, with $y\,\dot{s}\,I = y'\,\dot{s}\,I$ if and only if $y^{-1}y' \in U_{-\alpha,0+}$.

(2) By~(1), $U(F)\,w\,IsI = \bigcup_{y\in U_{-\alpha,0}} U(F)\,\dot{w}\,y\,\dot{s}\,I$. We write $\dot{w}\,y\,\dot{s} = (\dot{w}\,y\,\dot{w}^{-1})\,\dot{w}\dot{s}$ with $\dot{w}\,y\,\dot{w}^{-1}\in U_{\barr{w}(-\alpha)}(F)$. If $\barr{w}(\alpha)\in\Phi^-$, then $\barr{w}(-\alpha)\in\Phi^+$, so $\dot{w}\,y\,\dot{w}^{-1}$ is absorbed into $U(F)$ and every coset lies in $U(F)\,ws\,I$. If $\barr{w}(\alpha)\in\Phi^+$, the coset $y=1$ gives $U(F)\,ws\,I$. For $y$ in a non-trivial coset, the open-cell factorizations in the proof of Proposition~\ref{prop:ranksum} write the image of $y\,\dot{s}$ in $\mathsf{G}_v(k_F)$ as $\barr{u}\,\barr{b}$, where $\barr{u}$ is the image of some $u\in U_{\alpha,0}$ and $\barr{b}\in\barr{\mathsf{B}}(k_F)$. Then $u^{-1}\,y\,\dot{s}$ lies in the preimage $I$ of $\barr{\mathsf{B}}(k_F)$, so $y\,\dot{s}\in U_{\alpha,0}\cdot I$. Then $\dot{w}\,y\,\dot{s}\in(\dot{w}\,U_{\alpha,0}\,\dot{w}^{-1})\,\dot{w}\,I\subseteq U(F)\,\dot{w}\,I$.
\end{proof}

\subsection{The Borel--Casselman isomorphism}\label{ssec:BC}

Borel~\cite[Lem.~4.7]{Borel} proved the following for \emph{admissible} representations of semisimple groups. In the generality of smooth representations, one can invoke the Hecke-algebra localization of Bushnell. For an Iwahori subgroup the canonical map $V^I \to V_{\barr{U}}^{M_1}$ is an isomorphism, since the relevant Bernstein element acts invertibly (\cite[\S 3, Prop.~5, Thm.~1, Rem.~5--6]{Bushnell01}).

\begin{theorem}[Borel--Casselman--Bernstein, {\cite{Bushnell01}}]\label{thm:BC}
Let $(\pi,V)$ be a smooth representation of $G(F)$. The natural projection $V \to V_{\barr{U}}$ induces an isomorphism of $\AAA \cong \CC[\Omega_M]$-modules
\[
V^I \xrightarrow{\ \sim\ } V_{\barr{U}}^{M_1},
\]
where the $\AAA$-action on the right is induced by $M(F)/M_1 = \Omega_M$ in view of the embedding in Remark~\ref{rem:comm}.
\end{theorem}

%%% ------------------------------------------------------------------
\section{The reductive quotient at a special vertex}\label{sec:finite}
%%% ------------------------------------------------------------------

Recall the reductive quotient $\mathsf{G}_v(k_F) = K_v/K_v^+$, where $\mathsf{G}_v$ is a connected reductive group over the finite field~$k_F$, in particular quasi-split by Lang's theorem~\cite{Lang}.

Write $\mathsf{S}$ for the image of the maximal split torus~$S$ under the reduction $K_v\to\mathsf{G}_v(k_F)$. Since $v$ is special, $\mathsf{S}$ is a maximal split torus of $\mathsf{G}_v$. Moreover, the vertex $v$ is special precisely when the Weyl group of the relative root system $\Phi(\mathsf{G}_v,\mathsf{S})$ is the full relative Weyl group~$W_0$, so the simple reflections of $W_0$ correspond to the simple roots of $\mathsf{G}_v$. Nevertheless, the relative root system $\Phi(\mathsf{G}_v,\mathsf{S})$ need not coincide with $\Phi(G,S)$. For $\alpha\in\Delta$, let $\mathsf{U}_{\pm\alpha} = U_{\pm\alpha,0}/U_{\pm\alpha,0+}$ be the images in $\mathsf{G}_v$ of the depth-zero root groups. These filtration quotients are the root groups of $\mathsf{G}_v$ associated to the simple roots.

Let $\mathsf{S}_\alpha = (\ker\alpha)^\circ\subseteq\mathsf{S}$ be the codimension-one subtorus on which $\alpha$ vanishes. By the \emph{rank-one subgroup of $\mathsf{G}_v$ attached to~$s_\alpha$} we mean the derived group of the centralizer $Z_{\mathsf{G}_v}(\mathsf{S}_\alpha)$, equivalently the subgroup $\langle\mathsf{U}_\alpha,\mathsf{U}_{-\alpha}\rangle$ generated by the two opposite root groups. The character sum of Section~\ref{ssec:generator} is computed inside it.

We record what we have established in this discussion and in Section~\ref{ssec:Iwahori}.

\begin{fact}\label{fact:finite}
The reductive quotient $\mathsf{G}_v$ is connected reductive and quasi-split over~$k_F$, with relative Weyl group $W_0$. The image of $U$ is the unipotent radical $\mathsf{U}$ of the Borel $\mathsf{B} = \mathsf{M}\mathsf{U}$, and the image of $I$ is the opposite Borel $\barr{\mathsf{B}} = \mathsf{M}\barr{\mathsf{U}}$. For each simple root $\alpha \in \Delta$, with corresponding simple reflection $s_\alpha \in W_0$, the reduction map identifies $U_{\pm\alpha,0}/U_{\pm\alpha,0+}$ with the root groups of the rank-one subgroup of $\mathsf{G}_v$ attached to $s_\alpha$.
\end{fact}

The rank-one subgroup attached to a simple reflection $s_\alpha$ is of one of exactly two types. Its simply connected central cover is $\Res_{k_\alpha/k_F}G'$, where $k_\alpha$ is a finite extension of $k_F$ and $G'$ is absolutely simple. The group $G'$ is either $\mathrm{SL}_2$, of type $A_1$, or the quasi-split special unitary group $\mathrm{SU}_3$ of type ${}^2A_2$. In the latter case $\mathrm{SU}_3$ is attached to a quadratic extension $k_\alpha'/k_\alpha$ of finite fields and its relative root system is the non-reduced $BC_1$.

\smallskip
\noindent\textbf{Case R:} $G' = \mathrm{SL}_2$, so the rank-one subgroup is a central quotient of $\Res_{k_\alpha/k_F}\mathrm{SL}_2$ and its positive root group $\mathsf{U}_\alpha(k_F)$ is the additive group of~$k_\alpha$.

\smallskip
\noindent\textbf{Case NR:} $G' = \mathrm{SU}_3$, so the rank-one subgroup is a central quotient of $\Res_{k_\alpha/k_F}\mathrm{SU}_3$ and its positive root group $\mathsf{U}_\alpha(k_F)$ is a finite Heisenberg-type group.

\begin{remark}\label{rem:ramified}
The case is determined by $\mathsf{G}_v$ and not by whether $2\alpha \in \Phi = \Phi(G, S)$. Case NR requires a non-trivial quadratic extension of residue fields. Both cases occur for the quasi-split unitary groups $G = \mathrm{SU}_{2m+1}$ attached to a separable quadratic extension $E/F$, with relative root system $BC_m$ and one multipliable simple root~$\alpha$. If $E/F$ is \emph{unramified}, the residue extension $k_E/k_F$ is quadratic and the hyperspecial fiber $\mathsf{G}_v$ is the finite unitary group of type ${}^2A_{2m}$. The reflection $s_\alpha$ is of Case NR. If $E/F$ is \emph{ramified}, the residue extension is trivial. For odd residue characteristic, the two special vertices (the extreme vertices of the type-$BC_m$ affine diagram) then have reduced special fibers (\cite[Ex.~15.2.26]{KP}): the split $\mathrm{SO}_{2m+1}$ (type $B_m$) and $\mathrm{Sp}_{2m}$ (type $C_m$). At the $\mathrm{SO}_{2m+1}$ vertex $\alpha$ reduces to the short root of $B_m$ and $U_{2\alpha}$ reduces trivially ($U_{2\alpha,0} = U_{2\alpha,0+}$). There $s_\alpha$ is of Case R even though $2\alpha \in \Phi$. At the $\mathrm{Sp}_{2m}$ vertex $\alpha$ reduces to the long root of $C_m$ and its depth-zero reduction collapses into $U_{2\alpha}$, so that vertex violates Condition~\ref{cond:noncentral}.
\end{remark}

\begin{lemma}\label{lem:vertex}
For every connected reductive $G$ there is a special vertex $v$ in the apartment of~$S$ satisfying Condition~\ref{cond:noncentral}.  At such a vertex the reduction of Fact~\ref{fact:finite} identifies $U_{\alpha,0}/U_{\alpha,0+}$ with the root group of $\mathsf{G}_v$ of weight $\barr{\alpha}$ (rather than $2\barr{\alpha}$) for each multipliable simple~$\alpha$ with reduction $\barr{\alpha}$.
\end{lemma}

\begin{proof}
When $\Phi$ is reduced the condition is empty and any special vertex serves. Assume $\Phi$ non-reduced. Each non-reduced ($BC$) component of $\Phi$ contributes exactly one multipliable simple root, the short end root. The apartment of~$S$ in the adjoint building is the product of the component apartments, so the components may be treated independently. Fix a non-reduced component, with multipliable simple root~$\alpha_0$, and work in its apartment.

For every simple root $\alpha$ of the component there is an affine root with derivative exactly~$\alpha$ and not~$2\alpha$ (\cite[\S\S 1.2, 1.6]{Tits}). Let $h_\alpha$ be its vanishing hyperplane. Since the simple roots of the component form a basis of the dual of the apartment, the hyperplanes $\{h_\alpha\}_\alpha$ meet in a single point~$v$.

The derivatives of the affine roots vanishing at~$v$ form a root system (\cite[Prop.~9.4.23]{KP}). It contains every simple root of the component by construction, hence is stable under the component's Weyl group and contains every root direction, so $v$ is special. Moreover $v$ lies on~$h_{\alpha_0}$, so an affine root with derivative~$\alpha_0$ vanishes at~$v$. By the definition of the affine roots with derivative~$\alpha_0$ (\cite[\S 1.6]{Tits}) the image of $U_{\alpha_0,0}$ in $U_{\alpha_0}/U_{2\alpha_0}$ then strictly contains that of $U_{\alpha_0,0+}$, that is, the root $\alpha_0$ reduces non-centrally at~$v$. Taking a special vertex in every reduced component and the point just constructed in every non-reduced component yields the asserted special vertex of the full apartment.

For the final assertion, Fact~\ref{fact:finite} identifies $U_{\alpha,0}/U_{\alpha,0+}$ with a root group of $\mathsf{G}_v$, whose $\mathsf{S}$-weight is either $\barr{\alpha}$ or $2\barr{\alpha}$ (\cite[Prop.~9.4.23]{KP}). Non-central reduction gives $U_{\alpha,0}/U_{\alpha,0+}$ a non-zero image in $U_\alpha/U_{2\alpha}$, on which $S$ acts through~$\alpha$, so the weight is~$\barr{\alpha}$.
\end{proof}

%%% ------------------------------------------------------------------
\section{The Gelfand--Graev representation}\label{sec:GG}
%%% ------------------------------------------------------------------

\subsection{Depth-zero non-degenerate characters}\label{ssec:psi}

\begin{definition}\label{def:psi}
Let $\psi \colon U(F) \to \CC^\times$ be a smooth character.
\begin{enumerate}[\upshape(1)]
\item We call $\psi$ \emph{non-degenerate} if, for each simple root $\alpha\in\Delta$, it is non-trivial on $U_\alpha(F)$.
\item We call $\psi$ \emph{depth-zero non-degenerate at~$v$} if it is trivial on $U_{0+} = I\cap U$ and, for each simple root $\alpha \in \Delta$, its restriction to $U_{\alpha,0}$ is non-trivial on the quotient $U_{\alpha,0}/U_{\alpha,0+}$.
\end{enumerate}
\end{definition}

In particular, a depth-zero non-degenerate character is non-degenerate. Proposition~\ref{prop:depthzero} provides the converse at a vertex adapted to~$\psi$.

A character of $U(F)$ factors through the abelianization of $U(F)$ in characteristic zero (\cite[Thm.~4.1, Cor.~5.3]{BHder}):
\begin{equation}\label{eq:abel}
U(F)/[U(F),U(F)] \;\cong\; \prod_{\alpha\in\Delta} U_\alpha(F)/U_{2\alpha}(F).
\end{equation} 
Thus a character of $U(F)$ is determined by its restrictions to the factors of~\eqref{eq:abel}. By Fact~\ref{fact:finite}, the restriction of a depth-zero non-degenerate $\psi$ to $U_0$ descends to a non-degenerate character of $\mathsf{U}(k_F)$, again denoted $\psi$, non-trivial on each simple root group $\mathsf{U}_\alpha(k_F)$. Only this reduction enters the $I$-fixed computation below.

Existence and the central-stabilizer characterization below are standard for the Gelfand--Graev representation (cf.~\cite{BH}), and the stabilizer condition in (1) is the definition of non-degeneracy in~\cite{MP}. We record the \emph{depth-zero} realization at the vertex of Lemma~\ref{lem:vertex}.

\begin{lemma}\label{lem:exists}
\begin{enumerate}[\upshape(1)]
\item A smooth character of $U(F)$ is non-degenerate if and only if its stabilizer in $S(F)$ is contained in $Z(G)$.
\item A depth-zero non-degenerate character of $U(F)$ exists for every connected reductive group~$G$, and its stabilizer in $S(F)$ is central.
\end{enumerate}
\end{lemma}

\begin{proof}
(1) This is the equivalence of conditions (1) and (3) in the Proposition of \cite[\S 1.2]{BH}, whose definition of non-degeneracy agrees verbatim with Definition~\ref{def:psi}(1).

(2) By~\eqref{eq:abel} it suffices to choose, for each $\alpha \in \Delta$, a smooth character $\chi_\alpha$ of $U_\alpha(F)/U_{2\alpha}(F)$ trivial on the image of $U_{\alpha,0+}$ and non-trivial on the image of $U_{\alpha,0}$. These images are distinct by Condition~\ref{cond:noncentral} (automatically so when $2\alpha\notin\Phi$), so such a $\chi_\alpha$ exists by Pontryagin duality, since $U_{\alpha,0}$ is compact open. The resulting character of $U(F)$ is depth-zero non-degenerate, hence non-degenerate, and its stabilizer is central by~(1).
\end{proof}

\begin{proposition}\label{prop:depthzero}
Every non-degenerate character $\psi$ of $U(F)$ is depth-zero non-degenerate at a special vertex $v_\psi$ of the apartment of~$S$ satisfying Condition~\ref{cond:noncentral}. In particular Corollary~\ref{cor:main} applies to every non-degenerate~$\psi$.
\end{proposition}

\begin{proof}
Fix $\alpha\in\Delta$ and let $x$ vary in the apartment of~$S$ in the adjoint building. The restriction $\psi_\alpha$ of $\psi$ to $U_\alpha(F)/U_{2\alpha}(F)$ is a non-trivial smooth character, so it is trivial on the image of the Moy--Prasad subgroup $U_\alpha(F)_{x,r}$ for $r$ large and non-trivial for $r$ small. These images depend on~$x$ only through $\langle\alpha,x\rangle$, and for fixed~$x$ they change only when $r$ crosses the value at~$x$ of an affine root with derivative~$\alpha$ (\cite[\S 1.6]{Tits}). Hence the locus of those~$x$ at which $\psi_\alpha$ is non-trivial on the image of $U_\alpha(F)_{x,0}$ and trivial on the image of $U_\alpha(F)_{x,0+}$ is a single affine root hyperplane~$h_\alpha$.

By the intersection argument in the proof of Lemma~\ref{lem:vertex}, applied componentwise, $v_\psi := \bigcap_{\alpha\in\Delta}h_\alpha$ is a special vertex satisfying Condition~\ref{cond:noncentral}. The character $\psi$ is depth-zero non-degenerate at~$v_\psi$. Indeed, for each $\alpha\in\Delta$ it is non-trivial on $U_{\alpha,0}/U_{\alpha,0+}$ and trivial on $U_{\alpha,0+}$ because $v_\psi$ lies on~$h_\alpha$. It is trivial on the non-simple positive root groups by~\eqref{eq:abel}, hence on $U_{0+}$. Corollary~\ref{cor:main} is proved at an arbitrary special vertex satisfying Condition~\ref{cond:noncentral}, so it applies with $v = v_\psi$.
\end{proof}

The \emph{Gelfand--Graev representation} is $V = \ind_U^G \psi$, the space of smooth functions $f$ on $G(F)$ with $f(ug) = \psi(u)f(g)$ for all $u \in U(F)$, $g \in G(F)$, compactly supported modulo~$U(F)$.

\subsection{The Jacquet module}\label{ssec:Jacquet}

The computation of $V^I$ in this subsection and the next follows Chan--Savin closely. Indeed, Lemma~\ref{lem:open_cell}, Proposition~\ref{prop:Jacquet} and Lemma~\ref{lem:generator} are the relative forms of \cite[Lem.~4.1, Prop.~4.2, Lem.~4.3]{CS} respectively. The dictionary replaces the maximal split torus of the split case by the minimal Levi~$M$, the Bruhat cells of the Borel subgroup by the relative Bruhat cells of~$P$, and the hyperspecial vertex by the special vertex of Section~\ref{ssec:Iwahori}. The new point is Case~NR in the proof of Proposition~\ref{prop:ranksum}, a character sum in a finite special unitary group with no split counterpart.

Consider the relative Bruhat decomposition, in the opposite-parabolic form
\[
G(F) = \bigsqcup_{w \in W_0} U(F)\,\dot{w}\,M(F)\,\barr{U}(F)
= \bigsqcup_{w \in W_0} P(F)\,\dot{w}\,\barr{P}(F),
\]
whose open cell, at $w = 1$, is $U(F)M(F)\barr{U}(F)$. For $w \in W_0$ set $X_w = U(F)\,\dot{w}\,M(F)\,\barr{U}(F)$ and let $V_r$ be the subspace of $V$ of functions supported on $\bigcup_{\ell(w)\le r} X_w$. The closure of $X_w$ meets only cells of length at least $\ell(w)$, so $\bigcup_{\ell(w)\le r}X_w$ is open and each $V_r$ is a $\barr{P}(F)$-submodule of~$V$. Let $V_w$ be the space of smooth left $(U,\psi)$-equivariant functions on $X_w$, compactly supported modulo~$U(F)$.

\begin{lemma}\label{lem:open_cell}
The inclusion $V_0 \hookrightarrow V$ induces an isomorphism $(V_0)_{\barr{U}} \xrightarrow{\sim} V_{\barr{U}}$ of $M(F)$-modules.
\end{lemma}

\begin{proof}
For $r \ge 1$, restriction of functions to the cells of length~$r$ gives a short exact sequence of $\barr{P}(F)$-modules
\[
0 \to V_{r-1} \to V_r \to \bigoplus_{\ell(w)=r} V_w \to 0.
\]
The Jacquet functor is exact, so it suffices to prove $(V_w)_{\barr{U}} = 0$ for $\ell(w) > 0$. For this it suffices to find, for each $f \in V_w$, a compact open subgroup $\barr{U}_1 \subseteq \barr{U}(F)$ with
\[
\int_{\barr{U}_1} f(x\,\barr{u})\,d\barr{u} = 0 \qquad\text{for all } x \in X_w.
\]

Fix $w$ with $\ell(w) > 0$ and $f \in V_w$. Choose a simple root $\alpha \in \Delta$ with $w^{-1}(\alpha) \in \Phi^-$. Then $A := \dot{w}^{-1}\,U_{\alpha,0}\,\dot{w}$ is a compact subgroup of $\barr{U}(F)$. The support of $f$ is compact modulo $U(F)$, hence contained in $U(F)\,\dot{w}\,C\,\barr{U}_f$ with $C \subset M(F)$ and $\barr{U}_f \subset \barr{U}(F)$ both compact. Choose a compact open subgroup $\barr{U}_1 \subseteq \barr{U}(F)$ containing $\barr{U}_f$ with $m\,\barr{U}_1\,m^{-1} \supseteq A$ for every $m \in C$, which is possible since $A$ and $C$ are compact.

Now, we show that the integral of $f(x\,\barr{u})$ over $\barr{u}\in\barr{U}_1$ vanishes for every $x \in X_w$. The integral is zero unless $x\,\barr{U}_1$ meets the support of~$f$, in which case $x = u\,\dot{w}\,m\,\barr{u}_0$ with $u \in U(F)$, $m \in C$, $\barr{u}_0 \in \barr{U}_1$. In view of the $(U,\psi)$-equivariance of $f$ and the substitution $\barr{u} \mapsto \barr{u}_0\,\barr{u}$, it reduces to the vanishing of the integral for $x = \dot{w}\,m$. Conjugating the variable by~$m$ replaces the domain by $m\barr{U}_1 m^{-1}$ and the integrand by $f(\dot{w}\,\barr{u}'\,m)$, up to some constant. The resulting integral is unchanged under the substitutions $\barr{u}' \mapsto a\,\barr{u}'$ for $a \in A$, a subgroup of $m\barr{U}_1 m^{-1}$, so it equals its own average over $a \in A$. Since $\dot{w}\,a\,\dot{w}^{-1} \in U_{\alpha,0} \subseteq U(F)$, equivariance gives $f(\dot{w}\,a\,\barr{u}'\,m) = \psi(\dot{w}\,a\,\dot{w}^{-1})\,f(\dot{w}\,\barr{u}'\,m)$, so averaging over $a$ first multiplies the integral by the average of $\psi$ over $U_{\alpha,0}$, which is zero because $\psi$ is non-trivial on the compact group $U_{\alpha,0}$ (Definition~\ref{def:psi}(2)). Hence the integral vanishes and $(V_w)_{\barr{U}} = 0$.
\end{proof}

\begin{proposition}\label{prop:Jacquet}
There is an isomorphism of $M(F)$-modules $\mathcal{J} \colon V_{\barr{U}} \xrightarrow{\sim} C_c^\infty(M(F))$.
\end{proposition}

\begin{proof}
By Lemma~\ref{lem:open_cell} it suffices to treat $(V_0)_{\barr{U}}$. The multiplication map $U(F)\times M(F)\times\barr{U}(F) \to G(F)$ is a homeomorphism onto the open cell, so every $f \in V_0$ satisfies $f(u\,m\,\barr{u}) = \psi(u)\,f(m\,\barr{u})$. Restriction to $M(F)\,\barr{U}(F)$ then identifies $V_0$ with $C_c^\infty(M(F)\times\barr{U}(F))$. Under this identification $\barr{U}(F)$ acts by right translation in the variable $\barr{u}$, so integration in that variable computes the coinvariants. With the measure on $\barr{U}(F)$ normalized by $\vol(\barr{U}\cap I) = 1$, the map
\[
\mathcal{J}(f)(m) = \delta_{\barr{P}}(m)^{1/2}\int_{\barr{U}(F)} f(m\barr{u})\,d\barr{u}
\]
descends to a linear isomorphism $(V_0)_{\barr{U}} \xrightarrow{\ \sim\ } C_c^\infty(M(F))$. The factor $\delta_{\barr{P}}^{1/2}$ makes it $M(F)$-equivariant for the normalized action on the source and right translation on the target.
\end{proof}

Combining Theorem~\ref{thm:BC} and Proposition~\ref{prop:Jacquet},
\[
V^I \cong V_{\barr{U}}^{M_1} \cong C_c^\infty(M(F))^{M_1}
\cong C_c(M(F)/M_1) \cong \CC[\Omega_M],
\]
where the last identification is induced by $M(F)/M_1 \xrightarrow{\ \sim\ } \Omega_M$ (Lemma~\ref{lem:M1}), under which $\AAA = \CC[\Omega_M]$ acts by translation. Hence $V^I$ is a free $\AAA$-module of rank one. Let $\mathrm{ch}_{M_1} \in C_c^\infty(M(F))$ be the characteristic function of~$M_1$. It corresponds to $1 \in \CC[\Omega_M]$ and generates this $\AAA$-module.

\subsection{The generator and its transformation law}\label{ssec:generator}

\begin{lemma}\label{lem:generator}
Let $\mathrm{ch}_I^\psi$ be the function on $G(F)$ supported on $U(F)\cdot(I\cap\barr{P})$ defined by $\mathrm{ch}_I^\psi(u\,i) = \psi(u)$ for $u \in U(F)$, $i \in I\cap\barr{P}$. Then:
\begin{enumerate}[\upshape(1)]
\item $\mathrm{ch}_I^\psi \in V_0 \cap V^I$.
\item $\mathcal{J}(\mathrm{ch}_I^\psi) = \mathrm{ch}_{M_1}$.
\end{enumerate}
\end{lemma}

\begin{proof}
The function is well-defined because $U \cap \barr{P} = \{1\}$. It is right $I$-invariant. Indeed, for $g = u\,i$ in the support ($u\in U(F)$, $i\in I\cap\barr{P}$) and $k\in I$, write $i\,k = u_0\,m\,\barr{u}$ with $u_0 \in U_{0+}$, $m \in M_1$, $\barr{u} \in \barr{U}_0$ (the factorization~\eqref{eq:iwahori} of $ik\in I$). Then $g\,k = (u\,u_0)\,(m\,\barr{u})$ again lies in $U(F)\cdot(I\cap\barr{P})$, and $\mathrm{ch}_I^\psi(g\,k) = \psi(u\,u_0) = \psi(u) = \mathrm{ch}_I^\psi(g)$, since $\psi$ is trivial on $I\cap U = U_{0+}$. The function lies in $V_0$ since $U(F)\cdot(I\cap\barr{P})$ lies in the open cell. For~(2), a point $m\,\barr{u}$ with $m \in M(F)$ and $\barr{u} \in \barr{U}(F)$ lies in the support $U(F)\cdot(I\cap\barr{P})$ only if $m \in M_1$ and $\barr{u} \in I\cap\barr{U}$, by uniqueness of the open-cell factorization. Hence $\mathcal{J}(\mathrm{ch}_I^\psi)$ vanishes off~$M_1$. For $m \in M_1$ the integrand is $1$ exactly on $I\cap\barr{U}$ and $\delta_{\barr{P}}(m)^{1/2} = 1$ since $M_1$ is compact. So $\mathcal{J}(\mathrm{ch}_I^\psi) = \mathrm{ch}_{M_1}$.
\end{proof}

\begin{proposition}\label{prop:ranksum}
$T_w \cdot \mathrm{ch}_I^\psi = (-1)^{\ell(w)}\,\mathrm{ch}_I^\psi$ for all $w \in W_0$.
\end{proposition}

\begin{proof}
It suffices to prove $T_{s}\cdot\mathrm{ch}_I^\psi = -\mathrm{ch}_I^\psi$ for each simple reflection $s = s_\alpha$. The function $T_{s}\cdot\mathrm{ch}_I^\psi$ is again left $(U,\psi)$-equivariant and right $I$-invariant, so by the Iwasawa decomposition~\eqref{eq:iwasawa} it is determined by its values at the representatives $\dot{w}$, $w\in\widetilde{W}$, which we evaluate by means of~\eqref{eq:Haction} and Lemma~\ref{lem:bruhat}. Because $\mathrm{ch}_I^\psi$ is supported on $U(F)\cdot(I\cap\barr{P})\subseteq U(F)\,I$, and $U(F)\,w\,IsI$ meets only the cells $ws$ and~$w$ (Lemma~\ref{lem:bruhat}\,(2)), the value $T_{s}\cdot\mathrm{ch}_I^\psi(w)$ can be non-zero only when one of $ws,\,w$ equals $1$ in $\widetilde{W}$. This forces $w \in \{1, s\}$.

At $w = s$, represented by $\dot{s}^{-1}$, conjugation by $\dot{s}$ carries $U_{-\alpha}$ to $U_\alpha$, and
\[
T_{s}\cdot\mathrm{ch}_I^\psi(s) = \sum_{y \,\in\, U_{-\alpha,0}/U_{-\alpha,0+}}
\mathrm{ch}_I^\psi(\dot{s}^{-1} y\,\dot{s})
= \sum_{u\in\mathsf{U}_\alpha(k_F)} \psi(u) = 0,
\]
by non-degeneracy of $\psi$ on $\mathsf{U}_\alpha$, where the $y$ are the coset representatives of Lemma~\ref{lem:bruhat}\,(1). At $w = 1$ we must evaluate
\begin{equation}\label{eq:mainsum}
T_{s}\cdot\mathrm{ch}_I^\psi(1) = \sum_{y \,\in\, U_{-\alpha,0}/U_{-\alpha,0+}}
\mathrm{ch}_I^\psi(y\,\dot{s}).
\end{equation}
Since $\mathrm{ch}_I^\psi$ takes the value $1$ at $w = 1$ and $0$ at $w = s$, the identity $T_s\cdot\mathrm{ch}_I^\psi = -\mathrm{ch}_I^\psi$ is equivalent to the sum~\eqref{eq:mainsum} being~$-1$, which we now establish.

The representatives $y\,\dot{s}$ lie in the special maximal parahoric $K_v$, and we reduce the sum to the finite quotient $\mathsf{G}_v(k_F)$. The support of $\mathrm{ch}_I^\psi$ meets $K_v$ in $U_0\cdot(I\cap\barr{P}) = U_0\,I$ (by~\eqref{eq:iwahori}), which is the full preimage of $\mathsf{U}(k_F)\,\barr{\mathsf{B}}(k_F)$ under $K_v \to \mathsf{G}_v(k_F)$, since $I$ is the preimage of $\barr{\mathsf{B}}(k_F)$ and $U_0$ surjects onto $\mathsf{U}(k_F)$ (Section~\ref{ssec:Iwahori}). There $\mathrm{ch}_I^\psi(u\,i) = \psi(u)$ with $u \in U_0$, which depends only on the reduction $\barr{u} \in \mathsf{U}(k_F)$ of~$u$. Hence $\mathrm{ch}_I^\psi|_{K_v}$ descends to the function $\barr{\phi}$ on $\mathsf{G}_v(k_F)$ supported on $\mathsf{U}(k_F)\,\barr{\mathsf{B}}(k_F)$ with $\barr{\phi}(\barr{u}\,\barr{b}) = \psi(\barr{u})$. The sum~\eqref{eq:mainsum} is therefore the finite sum
\[
\sum_{y \in \mathsf{U}_{-\alpha}(k_F)} \barr{\phi}(y\,\dot{s}),
\]
computed inside the rank-one subgroup of $\mathsf{G}_v$ generated by $\mathsf{U}_{\pm\alpha}$. This subgroup falls into Case~R or Case~NR of the classification of Section~\ref{sec:finite}. Each term is insensitive to the choice of~$\dot{s}$, since any two elements of $\mathsf{G}_v(k_F)$ normalizing $\mathsf{S}$ and inducing $s_\alpha$ differ by an element of $\mathsf{M}(k_F) \subseteq \barr{\mathsf{B}}(k_F)$, which $\barr{\phi}$ absorbs on the right. We treat the two cases in turn.

\smallskip
\noindent\textbf{Case R.} Here $\mathsf{U}_{-\alpha}(k_F) \cong k_\alpha$, and the computation is the split-case computation of Chan--Savin (\cite[proof of Lem.~4.3]{CS}). The rank-one subgroup is a central quotient of $\Res_{k_\alpha/k_F}\mathrm{SL}_2$, and neither operation affects the sum. Indeed, Weil restriction is invisible on $k_F$-points, and the central quotient is an isomorphism on the root groups. We compute in $\mathrm{SL}_2$ over $k_\alpha$ with the Weyl representative $\dot{s} = \bigl(\begin{smallmatrix}0&1\\-1&0\end{smallmatrix}\bigr)$. The standard identity
\[
\begin{pmatrix}1&0\\ \tau&1\end{pmatrix}
\begin{pmatrix}0&1\\ -1&0\end{pmatrix}
= \begin{pmatrix}1& \tau^{-1}\\ 0&1\end{pmatrix}
\begin{pmatrix}\tau^{-1}&0\\ 0&\tau\end{pmatrix}
\begin{pmatrix}1&0\\ -\tau^{-1}&1\end{pmatrix}
\qquad (\tau \in k_\alpha^\times)
\]
puts $y\,\dot{s}$ in the open cell with unipotent part $x_\alpha(\tau^{-1})$, while $\tau = 0$ gives the closed cell and contributes~$0$. Hence
\[
\sum_{y\in\mathsf{U}_{-\alpha}(k_F)} \barr{\phi}(y\,\dot{s})
= \sum_{\tau\in k_\alpha^\times} \psi(\tau^{-1})
= -1,
\]
since $\psi$ is a non-trivial additive character of~$k_\alpha$.

\smallskip
\noindent\textbf{Case NR.} Here the rank-one subgroup is a central quotient of $\Res_{k_\alpha/k_F}\mathrm{SU}_3$, where $\mathrm{SU}_3$ is the quasi-split special unitary group over $k_\alpha$ attached to a quadratic extension $k_\alpha'/k_\alpha$ of finite fields. As in Case~R, neither operation affects the sum, so we compute in $\mathrm{SU}_3(k_\alpha)$. Write $q_\alpha = |k_\alpha|$, so $|k_\alpha'| = q_\alpha^2$. Let $c\mapsto\barr{c}$ be the non-trivial element of $\Gal(k_\alpha'/k_\alpha)$ and $\Tr = \Tr_{k_\alpha'/k_\alpha}$. Every identity below is an identity between explicit matrices with entries in~$k_\alpha'$. Realize $\mathrm{SU}_3 = \{g\in\mathrm{SL}_3(k_\alpha') : \barr{g}^{\,T}Jg = J\}$ with the anti-diagonal hermitian form $J = \bigl(\begin{smallmatrix}0&0&1\\0&1&0\\1&0&0\end{smallmatrix}\bigr)$. Its root groups
\[
u_+(a,b) = \begin{pmatrix}1&a&b\\0&1&-\barr a\\0&0&1\end{pmatrix}
\ (b+\barr b+a\barr a = 0), \qquad
u_-(e,d) = \begin{pmatrix}1&0&0\\-\barr e&1&0\\d&e&1\end{pmatrix}
\ (d+\barr d+e\barr e = 0)
\]
are the groups $\mathsf{U}_{\pm\alpha}$ of Fact~\ref{fact:finite}, with $\mathsf{U}_{2\alpha} = \{u_+(0,b): b+\barr b = 0\}$. We note that $\mathsf{U}_{-\alpha}(k_F)$ is a Heisenberg-type group of order $q_\alpha^3$ and, for each fixed $e$, the admissible $d$ form a coset of the trace-zero line, hence of size $q_\alpha$. The character $\psi$ factors through the coordinate $u_+(a,b)\mapsto a$, a non-trivial character of $\mathsf{U}_\alpha/\mathsf{U}_{2\alpha}\cong k_\alpha'$. The diagonal torus $\{m(a) = \mathrm{diag}(\barr a^{-1},\,\barr a a^{-1},\,a) : a\in k_\alpha'^\times\}$ lies in $\mathsf{M}$, and we take the Weyl representative $\dot{s} = \bigl(\begin{smallmatrix}0&0&1\\0&-1&0\\1&0&0\end{smallmatrix}\bigr) \in\mathrm{SU}_3$. Direct multiplication gives
\[
u_-(e,d)\,\dot{s} =
\begin{pmatrix}0&0&1\\0&-1&-\barr e\\1&-e&d\end{pmatrix}.
\]
When $u_-(e,d)\,\dot{s}$ admits an open-cell factorization $u_+\cdot m\cdot u_-$ with $u_+\in\mathsf{U}_\alpha$, $m\in\mathsf{M}$, $u_-\in\mathsf{U}_{-\alpha}$, the corresponding term of~\eqref{eq:mainsum} equals $\psi(u_+)$, and when it admits none the term vanishes, since $\barr{\phi}$ is supported on $\mathsf{U}(k_F)\,\barr{\mathsf{B}}(k_F)$. We split the sum into three cases.

\emph{(i) $e=0$, $d=0$.} Here $u_-(0,0)\,\dot{s} = \dot{s}$ does not lie in the open cell, and the term vanishes.

\emph{(ii) $e=0$, $d\neq 0$ (so $d+\barr d = 0$).} Then $\barr d^{-1} = -d^{-1}$, and
\[
u_-(0,d)\,\dot{s} = \begin{pmatrix}0&0&1\\0&-1&0\\1&0&d\end{pmatrix}
= u_+(0,\,d^{-1})\cdot
\underbrace{\begin{pmatrix}\barr d^{-1}&0&0\\0&-1&0\\0&0&d\end{pmatrix}}_{m(d)\in\mathsf{M}}
\cdot\;u_-(0,\,d^{-1}).
\]
The $\mathsf{U}_\alpha$-component $u_+(0,d^{-1})$ lies in $\mathsf{U}_{2\alpha}$, on which $\psi$ is trivial, so each of the $q_\alpha - 1$ such terms contributes~$1$.

\emph{(iii) $e\neq 0$ (so $d+\barr d = -e\barr e\neq 0$ and $d\neq 0$).} Using the relation $d+\barr d+e\barr e = 0$, one checks the open-cell factorization
\[
u_-(e,d)\,\dot{s}
= u_+\bigl(e\barr d^{-1},\, d^{-1}\bigr)\cdot
\underbrace{\begin{pmatrix}\barr d^{-1}&0&0\\0&\barr d\,d^{-1}&0\\0&0&d\end{pmatrix}}_{m(d)\in\mathsf{M}}
\cdot\;u_-\bigl(-e d^{-1},\, d^{-1}\bigr),
\]
where $u_+(e\barr d^{-1}, d^{-1})\in\mathsf{U}_\alpha$ and $u_-(-ed^{-1}, d^{-1})\in\mathsf{U}_{-\alpha}$ amount to $(d+\barr d+e\barr e)/(d\barr d) = 0$. As $\psi$ is trivial on $\mathsf{U}_{2\alpha}$, the term is $\psi(e\barr{d}^{-1})$, and the contribution is
\[
S = \sum_{\substack{e\in k_\alpha'^\times\\ d:\; d+\barr{d}=-e\barr{e}}}
\psi\bigl(e\barr{d}^{-1}\bigr).
\]
The map $(e,d)\mapsto c := e\barr{d}^{-1}$ is exactly $q_\alpha$-to-one onto $k_\alpha'^\times$. Indeed, for any fixed $c \neq 0$, one has $e = c\barr{d}$, and dividing the constraint $d+\barr{d} = -c\barr{c}\,d\barr{d}$ by $d\barr d$ turns it into $\Tr(d^{-1}) = -c\barr{c}$, which has exactly $q_\alpha$ solutions. Hence
\[
S = q_\alpha \sum_{c\in k_\alpha'^\times} \psi(c) = q_\alpha(0-1) = -q_\alpha.
\]
Again, here we use the fact that $\psi$ is a non-trivial additive character of $k_\alpha'$.

\smallskip
The three contributions sum to $0 + (q_\alpha - 1) + (-q_\alpha) = -1$. In both Case~R and Case~NR the sum~\eqref{eq:mainsum} equals~$-1$, hence $T_s\cdot\mathrm{ch}_I^\psi = -\mathrm{ch}_I^\psi$.
\end{proof}

%%% ------------------------------------------------------------------
\section{Main theorem and consequences}\label{sec:main}
%%% ------------------------------------------------------------------

Throughout this section, $\psi$ is a depth-zero non-degenerate character of $U(F)$ at a special vertex~$v$ (Definition~\ref{def:psi}) and $I$ is the Iwahori attached to $(P,v)$ as in Section~\ref{ssec:Iwahori}.

\begin{corollary}\label{cor:main}
Let $V = \ind_U^G\psi$. Then:
\begin{enumerate}[\upshape(1)]
\item $V^I$ is a free $\AAA$-module of rank one, generated by $\mathrm{ch}_I^\psi$.
\item $V^I \cong \HH \otimes_{\HH_{W_0}} \sgn$ as $\HH$-modules.
\end{enumerate}
\end{corollary}

\begin{proof}
Part~(1) follows from Lemma~\ref{lem:generator}\,(2) and the identification $V^I \cong \CC[\Omega_M]$ of Section~\ref{ssec:Jacquet}. For~(2), Proposition~\ref{prop:ranksum} gives an element of $\Hom_{\HH_{W_0}}(\sgn, V^I)$ sending $1 \mapsto \mathrm{ch}_I^\psi$. By Frobenius reciprocity it extends to an $\HH$-module map $\HH\otimes_{\HH_{W_0}}\sgn \to V^I$ with $1\otimes 1 \mapsto \mathrm{ch}_I^\psi$. Both sides are free $\AAA$-modules of rank one, and the map sends generator to generator, so it is an isomorphism.
\end{proof}

The proof uses no hypothesis on the center of $G$ and no hypothesis on the residue characteristic. The depth-zero non-degenerate characters of $U(F)$ may form several orbits under $S(F)$, and Corollary~\ref{cor:main} holds for each.

\medskip
\noindent\textbf{Genericity criterion.} For a smooth representation $\pi$ of $G(F)$ define
\[
S(\pi) = \bigl\{v \in \pi^I : T_w\cdot v = (-1)^{\ell(w)}\,v
\ \text{for all } w \in W_0\bigr\}.
\]
Let $\pi_{U,\psi}$ be the unnormalized \emph{twisted Jacquet module}, the maximal quotient of $\pi$ on which $U(F)$ acts by~$\psi$.

The Iwahori $(I,\mathbf{1})$ is a type for a finite set of Bernstein components (\cite[Thm.~4.8]{Morris}), so the category $\mathrm{Rep}_I(G)$ of smooth representations generated by their $I$-fixed vectors is equivalent, via $V\mapsto V^I$, to the category of $\HH$-modules (\cite[Thm.~4.3(ii)]{BK}). These components are the connected components of $\operatorname{Spec} Z[\HH]$, where $Z[\HH]$ is the center of~$\HH$ (Bernstein~\cite{Bernstein}). Since $Z[\HH] = \AAA^{W_0} = \CC[\Omega_M]^{W_0}$ (\cite[\S 6]{Rostami}), these components are indexed by the $W_0$-orbits of characters of the torsion subgroup $\Omega_{M,\mathrm{tor}}$. This index set is stable under $\chi\mapsto\chi^{-1}$, and the
contragredient carries the component of $\chi$ to that of $\chi^{-1}$, so $\mathrm{Rep}_I(G)$ is closed under the
contragredient.

\begin{corollary}\label{cor:genericity}
Let $\pi\in\mathrm{Rep}_I(G)$. Then the canonical map $S(\pi) \to \pi_{U,\psi}$, obtained by composing $S(\pi)\hookrightarrow\pi$ with $\pi \twoheadrightarrow \pi_{U,\psi}$, is a bijection. In particular, $\pi$ is $\psi$-generic ($\pi_{U,\psi}\neq 0$) if and only if $S(\pi)\neq 0$, equivalently, $\pi^I$ contains a nonzero vector transforming by the sign character~$\sgn$ of $\HH_{W_0}$.
\end{corollary}

\begin{proof}
We argue as in Chan--Savin~\cite[Corollary~4.5]{CS}. We indicate the points requiring care in the present generality. Let $\tilde\psi = \psi^{-1}$, again depth-zero non-degenerate at~$v$, so Corollary~\ref{cor:main} applies to $\ind_U^G\tilde\psi$. It suffices to show the dual map $(\pi_{U,\psi})^* \to S(\pi)^*$ is a bijection, and as in \emph{loc.\ cit.}
\begin{align*}
(\pi_{U,\psi})^*
&\cong \Hom_G(\pi,\, \Ind_U^G\psi)
\cong \Hom_G(\ind_U^G\tilde\psi,\, \tilde\pi) \\
&\cong \Hom_\HH\bigl(\HH\otimes_{\HH_{W_0}}\sgn,\, \tilde\pi^I\bigr)
\cong \Hom_{\HH_{W_0}}(\sgn,\, \tilde\pi^I) \cong S(\tilde\pi) \cong S(\pi)^*,
\end{align*}
by the duality $(\ind_U^G\tilde\psi)^\vee \cong \Ind_U^G\psi$ and Frobenius reciprocity. The explicit computation in the proof of \emph{loc.\ cit.} shows that this composite realizes the dual of the canonical map $S(\pi) \to \pi_{U,\psi}$. It applies verbatim here, with the quotient measure on $U(F)\backslash G(F)$ induced by the normalizations $\vol(I) = \vol(U(F)\cap I) = 1$.

Two points deserve mention. First, $\ind_U^G\tilde\psi$ does not lie in $\mathrm{Rep}_I(G)$, so for the third isomorphism one interposes the Bernstein projection $p_I$ onto $\mathrm{Rep}_I(G)$. Since $\tilde\pi\in\mathrm{Rep}_I(G)$ (closure under the contragredient, above), one has $\Hom_G(\ind_U^G\tilde\psi,\tilde\pi) = \Hom_G(p_I\,\ind_U^G\tilde\psi,\tilde\pi)$, while $(p_I\,\ind_U^G\tilde\psi)^I = (\ind_U^G\tilde\psi)^I$.

Second, we explain the last isomorphism. The Iwahori idempotent splits $\pi = \pi^I\oplus\pi(I)$, so $\tilde\pi^I = (\pi^I)^*$ with no admissibility hypothesis. For any smooth representation $\sigma$ of $G(F)$ one has $S(\sigma) = e_{\sgn}\,\sigma^I$, where
\[
e_{\sgn} = \Bigl(\sum_{w\in W_0} q_w^{-1}\Bigr)^{-1}\sum_{w\in W_0}(-1)^{\ell(w)}\,q_w^{-1}\,T_w
\]
is the sign idempotent of $\HH_{W_0}$, satisfying $T_s\,e_{\sgn} = -e_{\sgn}$. Let $\iota$ be the anti-involution of $\HH$ defined by $\iota(f)(x) = f(x^{-1})$. We note $\iota(e_{\sgn}) = e_{\sgn}$. Each $f \in \HH$ acts on $\tilde\pi^I = (\pi^I)^*$ as the transpose of $\pi(\iota(f))$. Hence $S(\tilde\pi) = e_{\sgn}\,\tilde\pi^I = (e_{\sgn}\,\pi^I)^* = S(\pi)^*$.
\end{proof}

\medskip
\noindent\textbf{Uniqueness of the isomorphism.} Any $\HH$-module automorphism of $\HH\otimes_{\HH_{W_0}}\sgn$ is multiplication by a unit of $Z[\HH]$. If $G$ is semisimple with $\Omega_M$ torsion-free (for instance, when $G$ is split semisimple or simply connected), then $Z[\HH]^\times = \CC^\times$ and the isomorphism of Corollary~\ref{cor:main}\,(2) is unique up to a non-zero scalar. These statements are proved as in the uniqueness theorem of Chan--Savin \cite[\S 4.3]{CS}, with the lattice $X$ there replaced by $\Omega_M$ and the commutation relations by the Bernstein relations of \cite[\S 5.4]{Rostami}. We omit the details.

\begin{remark}\label{rem:PGL2D}
Torsion in $\Omega_M$ obstructs uniqueness even for semisimple~$G$. For $G = \mathrm{PGL}_2(D)$, with $D$ a central division algebra over $F$ of degree $d \ge 2$, the minimal Levi gives $\Omega_M \cong (\ZZ\oplus\ZZ)/\langle(d,d)\rangle$ with $W_0 = S_2$ swapping the coordinates. The $W_0$-fixed torsion class $t = [(1,1)]$, of order~$d$, makes $\Theta_t$ a non-scalar unit of $Z[\HH]$.
\end{remark}

\appendix

%%% ------------------------------------------------------------------
\section{Comparison with Solleveld--Opdam}\label{app:Sol}
%%% ------------------------------------------------------------------

We compare our explicit generator with the general framework of
Solleveld--Opdam \cite{SO26}. Its principal-series calculation and
Whittaker normalization come from Solleveld
\cite[Lemma~3.1 and Theorem~2.7]{Sol25}. Put $V=\ind_U^G\psi$, and let
$M^1$ be the subgroup
\[
M^1=
\bigcap_{\chi\in X_F^*(M)}
\ker\bigl(|\chi|_F\colon M(F)\longrightarrow\mathbb R_{>0}\bigr),
\]
where $X_F^*(M)$ is the group of $F$-rational algebraic characters of
$M$. We note that $M^1$ is also
the subgroup generated by all compact subgroups.

Recall that $\Omega_M=M(F)/M_1$.  The inclusion $M^1\subset M(F)$
identifies
\[
M^1/M_1=\Omega_{M,\mathrm{tor}}.
\tag{A.1}
\]
Indeed, every compact subgroup has finite image in the discrete group
$\Omega_M$.  Conversely, if $mM_1$ has finite order in $\Omega_M$, the
inverse image of the finite cyclic group generated by $mM_1$ is a compact subgroup of $M(F)$ containing~$m$.

Let $\displaystyle \Omega_{M,\mathrm{tor}}^\vee =\Hom(\Omega_{M,\mathrm{tor}},\CC^\times)$. View $\tau\in\Omega_{M,\mathrm{tor}}^\vee$ as a character of $M^1$
trivial on $M_1$.  Choose a unitary extension $\chi_\tau$ to $M(F)$ and put
$\mathfrak s_\tau=[M,\chi_\tau]_G$.  Different extensions differ by an
unramified character and determine the same inertial class.  Since
$U\cap M=\{1\}$ and $\chi_\tau$ is one-dimensional,
$\mathfrak s_\tau$ is simply generic in the sense of \cite[\S 6 and
Appendix~A]{SO26}.  The associated progenerator is
\[
\Pi_\tau
=I_P^G\bigl(\ind_{M^1}^{M(F)}\tau\bigr).
\]
Theorem~6.1(b), equation~(8.8), Corollary~8.7, and Theorem~A.1 of
\cite{SO26} identify $\Hom_{G(F)}(\Pi_\tau,V)$ with the module induced
from the Steinberg representation of the finite part of the corresponding
extended affine Hecke algebra. Let $\xi_\tau\in\Hom_{G(F)}(\Pi_\tau,V)$
be the vector corresponding to the standard tensor $1\otimes1$ in this
induced module. We call $\xi_\tau$ the distinguished vector attached
to~$\tau$.

Second adjointness, Frobenius reciprocity and
Proposition~\ref{prop:Jacquet} identify
{\small
\[
\Hom_{G(F)}(\Pi_\tau,V)
\cong\Hom_{M(F)}\bigl(\ind_{M^1}^{M(F)}\tau,V_{\barr U}\bigr)
\cong\Hom_{M^1}\bigl(\tau,C_c^\infty(M(F))\bigr).
\tag{A.2}
\]
}
With right translation on $C_c^\infty(M(F))$, the canonical vector on
the right is
\[
F_\tau(m)=
\begin{cases}
\tau(m),&m\in M^1,\\
0,&m\notin M^1.
\end{cases}
\tag{A.3}
\]
The Jacquet isomorphism underlying this construction and the map
$\mathcal J$ of Proposition~\ref{prop:Jacquet} may differ by a non-zero
scalar arising from the choice of Haar measure. This is a single scalar independent of
$\tau$.  We normalize the comparison so that $\xi_\tau$ corresponds
to $F_\tau$.  Without this normalization, all the
vectors below are multiplied by the same non-zero scalar.

Let
\[
\kappa\colon V^I\xrightarrow{\ \sim\ }
V_{\barr U}^{M_1}\xrightarrow[\mathcal J]{\ \sim\ }
C_c(M(F)/M_1)
\tag{A.4}
\]
be the composite of Theorem~\ref{thm:BC} and
Proposition~\ref{prop:Jacquet}.  For
$\omega\in\Omega_{M,\mathrm{tor}}$, write $\delta_\omega$ for the
characteristic function of the corresponding coset in $M(F)/M_1$.  Under
(A.2)--(A.4), the vector $\xi_\tau$ determines the unique
$v_\tau\in V^I$ satisfying
\[
\kappa(v_\tau)
=f_\tau
:=\sum_{\omega\in\Omega_{M,\mathrm{tor}}}
\tau(\omega)\delta_\omega.
\tag{A.5}
\]

\begin{proposition}\label{prop:Sol-recovery}
With the normalization above,
\[
\mathrm{ch}_I^\psi
=
\frac{1}{|\Omega_{M,\mathrm{tor}}|}
\sum_{\tau\in\Omega_{M,\mathrm{tor}}^\vee}v_\tau.
\tag{A.6}
\]
\end{proposition}

\begin{proof}
Write $\Omega :=\Omega_{M,\mathrm{tor}}$. By (A.5) and character
orthogonality,
\[
\frac{1}{|\Omega|}\sum_{\tau\in\Omega^\vee}f_\tau
=\sum_{\omega\in\Omega}
\left(
\frac{1}{|\Omega|}
\sum_{\tau\in\Omega^\vee}\tau(\omega)
\right)\delta_\omega
=\delta_1.
\]
The claim follows from $\kappa(v_\tau)=f_\tau$ and
$\kappa(\mathrm{ch}_I^\psi)=\mathrm{ch}_{M_1}=\delta_1$
(Lemma~\ref{lem:generator}\,(2)).
\end{proof}

When $G$ is unramified, its minimal Levi is an unramified torus and
$\Omega_{M,\mathrm{tor}}=1$.  Thus $v_1=\mathrm{ch}_I^\psi$.  This is the case used in
\cite{Luo26}.  In general, Proposition~\ref{prop:Sol-recovery} assembles
the vectors $v_\tau$ into the explicit Iwahori generator. This
still leaves a comparison between the normalized affine-Hecke generators
and the $T_s$ in our conventions.
Proposition~\ref{prop:ranksum} avoids that comparison by proving
$T_s\cdot\mathrm{ch}_I^\psi=-\mathrm{ch}_I^\psi$ directly in $V^I$.

%%% ------------------------------------------------------------------

\end{document}